\documentclass[12pt]{article}
\usepackage{amsmath,amsthm,amssymb}
\usepackage{hyperref}

\newtheorem{lemma}{Lemma}
\newtheorem{theorem}[lemma]{Theorem}

\begin{document}

	\begin{center}
		{\Large \bf Some Bounds on the Energy of Graphs with Self-Loops regarding $\lambda_{1}$ and $\lambda_{n}$}  \\
		\vspace{8mm}
		
		{\large \bf Minghua Li$^a\footnote{Corresponding author}$, Yue Liu$^{a}$}
		
		\vspace{6mm}
		
		\baselineskip=0.20in
		
		{\it $^a$School of Mathematics and Statistics, Fuzhou University, \\
Fuzhou, Fujian, 350108, P. R. China\/} \\
		{\rm E-mail:} {\tt 220320048@fzu.edu.cn, liuyue@fzu.edu.cn}\\[2mm]

		\vspace{5mm}
		
	\end{center}

\begin{abstract}
 Let $G_{S}$ be a graph with $n$ vertices obtained from a simple graph $G$ by attaching one self-loop at each vertex in $S \subseteq V(G)$. The energy of $G_{S}$ is defined by Gutman et al. as $E(G_{S})=\sum_{i=1}^{n}\left| \lambda_{i} -\frac{\sigma}{n} \right|$, where $\lambda_{1},\dots,\lambda_{n}$ are the adjacency eigenvalues of $G_{S}$ and $\sigma$ is the number of self-loops of $G_{S}$. In this paper, several upper and lower bounds of $E(G_{S})$ regarding $\lambda_{1}$ and $\lambda_{n}$ are obtained. Especially, the upper bound $E(G_{S}) \leq \sqrt{n\left(2m+\sigma-\frac{\sigma^{2}}{n}\right)}$ $(\ast)$ given by Gutman et al. is improved to the following bound
    \begin{align*}
        E(G_{S})\leq \sqrt{n\left(2m+\sigma-\frac{\sigma^{2}}{n}\right)-\frac{n}{2}\left(\left |\lambda_{1}-\frac{\sigma}{n}\right |-\left |\lambda_{n}-\frac{\sigma}{n}\right |\right)^{2}},
    \end{align*}
where $\left| \lambda_{1}-\frac{\sigma}{n}\right| \geq \dots \geq \left| \lambda_{n}-\frac{\sigma}{n}\right|$. Moreover, all graphs are characterized when the equality holds in Gutmans' bound $(\ast)$ by using this new bound.
\end{abstract}




\section{Introduction}

The $energy$ $E(G)$ of a simple graph $G$ was defined by Gutman \cite{Ref9} in 1978 as $E(G) = \sum_{i=1}^{n}\left|\lambda_{i}\right|$. Since then, graph energy has become an important topological index and played a crucial role in spectral graph theory and mathematical chemistry \cite{Ref4,Ref5}. More information and details about the graph energy are available in \cite{Ref12} and \cite{Ref18}. 

Since self-loops distinguish hetero-atoms from carbon atoms in heteroconjugated molecules\cite{Ref3}, graphs with self-loops have been applied to the topological studies on heteroconjugated molecules \cite{Ref10,Ref11,Ref22}.

Let $G_{S}$ be a graph obtained from a simple graph $G$ by attaching one self-loop at each vertex in $S$, where $S$ is a subset of the vertex set $V(G)=\left\{v_1,\dots,v_n\right\}$ and $\left| S \right|=\sigma$. 
The $adjacency$ $matrix$ $A(G_{S}) = [a_{ij}]_{n\times n}$ is a square symmetric matrix of order $n$ such that
\begin{align*}
a_{ij}=\begin{cases}
1 & \rm if \enspace \textit{i $\neq$ j} \enspace \rm{and} \enspace \textit v_{\textit i}, \textit v_{\textit j} \enspace are \enspace adjacent,\\
0 & \rm if \enspace \textit{i $\neq$ j} \enspace  \rm{and} \enspace \textit v_{\textit i}, \textit v_{\textit j} \enspace are \enspace not \enspace adjacent,\\
1 & \rm if \enspace \textit{i = j} \enspace \rm{and} \enspace \textit v_{\textit i} \in \textit S, \\
0 & \rm if \enspace \textit{i = j} \enspace \rm{and} \enspace \textit v_{\textit i} \notin \textit S.
\end{cases}
\end{align*}

Recently, the concept of graph energy has been extended from simple graphs to graphs with self-loops by Gutman et al. in \cite{Ref13}. The energy of $G_{S}$ is defined as 
\begin{align*}
    E(G_{S}) = \sum_{i=1}^{n}\left|\lambda_{i}-\frac{\sigma}{n}\right|,
\end{align*}
where $\lambda_{1},\dots,\lambda_{n}$ are the adjacency eigenvalues of $G_{S}$.

After the definition of $E(G_{S})$ was proposed, considerable attention has been paid to the energy of graphs with self-loops\cite{Ref1,Ref2,Ref16,Ref17,Ref21}.
In \cite{Ref16}, Jovanovi\'c et al. gave a simple set of graphs that satisfy $E(G_{S}) \leq E(G)$ and disproved the conjecture presented in \cite{Ref13}.
In \cite{Ref17}, Liu et al. derived several lower and upper bounds of $E(G_{S})$ by using the maximum degree $\Delta $, the minimum degree $\delta$ and the graph spread $s(G)$ of $G$.

In \cite{Ref13}, Gutman et al. gave the upper bound of $E(G_{S})$
\begin{align} \label{bound0}
    E(G_{S}) \leq \sqrt{n\left(2m+\sigma-\frac{\sigma^{2}}{n}\right)}. \tag{$\ast$}
\end{align}
The formula (\ref{bound0}) is a generalization of the original McClelland's bound $E(G) \leq \sqrt{2mn}$ (see \cite{Ref19}).
In \cite{Ref1}, Akbari et al. studied the bound (\ref{bound0}) and derived a necessary condition that $G_S$ is $(a, b)\mbox{-}$semi$\mbox{-}$regular when the equality holds.

In \cite{Ref20}, Milovanovi$\acute{\rm{c}}$ et al.  presented a upper bound of simple graph
\begin{align} \label{bound1}
    E(G)\leq \sqrt{2mn-\frac{n}{2}\left(\left |\lambda_{1}\right |-\left |\lambda_{n}\right |\right)^{2}},
\end{align}
and they also proved that the equality holds in (\ref{bound1}) if and only if $G = nK_{1}$ or $G =\frac{n}{2}K_{2}$ when $n$ is even. Motivated by the above result, we generalize the bound (\Ref{bound1}) to the graphs with self$\mbox{-}$loops.

The rest of this paper consists of two sections. In Section 2, the bound (\ref{bound0}) is improved to a new upper bound. In Section 3, several lower bounds are obtained.

\section{A better upper bound for the energy of graphs with self-loops}

In this section, our primary results are to prove a new upper bound, and characterize all graphs when the equality holds in Gutmans' bound (\ref{bound0}) by using this new upper bound and some other established conclusions. To begin with, a useful lemma is given for the eigenvalues of graphs with self-loops.

\begin{lemma} \cite{Ref13}  \label{lemma1}
    Let $G_{S}$ be a graph of order n, with m edges and $\sigma$ self-loops; $\lambda_{1},\dots,\lambda_{n}$ be its adjacency eigenvalues. Then
    \begin{align*}
     \sum_{i=1}^{n}\lambda_{i} = \sigma, \quad \sum_{i=1}^{n}\lambda_{i}^{2} = 2m+\sigma,
    \end{align*}
    and
        \begin{align*}
        \sum_{i=1}^{n}\left|\lambda_{i}-\frac{\sigma}{n}\right|^{2} = 2m+\sigma-\frac{\sigma^{2}}{n}.
    \end{align*}
\end{lemma}

In the following theorem, a new bound is obtained for the energy of graphs with self-loops. 
\begin{theorem} \label{theorem2}
    Let $G_{S}$ be a graph with $n \geq 2$ vertices, m edges and $\sigma$ self-loops; $\lambda_{1},\dots,\lambda_{n}$ be the adjacency eigenvalues of $G_{S}$ such that  
    $\left| \lambda_{1}-\frac{\sigma}{n}\right| \geq \dots \geq \left| \lambda_{n}-\frac{\sigma}{n}\right|.$
    Then
\begin{align} \label{bound2}
E(G_{S})\leq \sqrt{n\left(2m+\sigma-\frac{\sigma^{2}}{n}\right)-\frac{n}{2}\left(\left |\lambda_{1}-\frac{\sigma}{n}\right |-\left |\lambda_{n}-\frac{\sigma}{n}\right |\right)^{2}}.
\end{align}
\end{theorem}

\begin{proof}[Proof]
When $n=2$, $G_{S}=K_{2}$, $\widehat{K_{2}}$, $\widetilde{K_{2}}$, respectively. $\widehat{K_{2}}$ denotes $K_{2}$ attaching a self-loop to each of its vertices and $\widetilde{K_{2}}$ denotes $K_{2}$ attaching a self-loop to one of its vertices. It is easy to verify the bound (\ref{bound2}) holds when $n=2$. In the following proof, we assume $n \geq 3$.

From the Lagrange's identity (see \cite[Theorem 2.1.(2)]{Ref7})
\begin{align*}
    \left(\sum_{i=1}^{n}a_{i}^{2}\right)\left(\sum_{i=1}^{n}b_{i}^{2}\right)-\left(\sum_{i=1}^{n}a_{i}b_{i}\right)^{2}=\sum_{1 \leq i < j \leq n}\left( a_{i}b_{j}-a_{j}b_{i} \right)^{2},
\end{align*}
for $i=1,2,\dots,n$, let $a_{i}=1$, $b_{i}=\left | \lambda_{i} - \frac{\sigma}{n} \right |$ and by Lemma \ref{lemma1}, we have
\begin{align}    \label{eq3}
    n\left( 2m+\sigma-\frac{\sigma^{2}}{n}\right)-E(G_{S})^{2} \nonumber &= n\sum_{i=1}^{n} \left| \lambda_{i}-\frac{\sigma}{n} \right|^{2} - \left(\sum_{i=1}^{n}\left| \lambda_{i}-\frac{\sigma}{n}\right| \right)^{2}  \nonumber \\
    &= \sum_{1 \leq i < j \leq n}\left( \left| \lambda_{i}-\frac{\sigma}{n} \right| - \left| \lambda_{j}-\frac{\sigma}{n} \right| \right)^{2}.
\end{align}
Only keep items with parameter $\left| \lambda_{1}-\frac{\sigma}{n}\right|$ or $\left| \lambda_{n}-\frac{\sigma}{n}\right|$, the following inequality is obtained

\begin{eqnarray}   \label{eq4}
    && \sum_{1 \leq i < j \leq n}\left( \left| \lambda_{i}-\frac{\sigma}{n} \right| - \left| \lambda_{j}-\frac{\sigma}{n} \right| \right)^{2}   \nonumber\\
    &\geq& \sum_{i=2}^{n-1}\left( \left(\left| \lambda_{1}-\frac{\sigma}{n} \right| - \left| \lambda_{i}-\frac{\sigma}{n} \right| \right)^{2}+\left(\left| \lambda_{i}-\frac{\sigma}{n} \right| - \left| \lambda_{n}-\frac{\sigma}{n} \right| \right)^{2}\right) \nonumber\\
    && +\left(\left| \lambda_{1}-\frac{\sigma}{n} \right| - \left| \lambda_{n}-\frac{\sigma}{n} \right| \right)^{2}. 
\end{eqnarray} 
By the Jensen's inequality, it follows that
\begin{eqnarray}    \label{eq5}
&&\sum_{i=2}^{n-1}\left( \left(\left| \lambda_{1}-\frac{\sigma}{n} \right| - \left| \lambda_{i}-\frac{\sigma}{n} \right| \right)^{2}+\left(\left| \lambda_{i}-\frac{\sigma}{n} \right| - \left| \lambda_{n}-\frac{\sigma}{n} \right| \right)^{2}\right)  \nonumber    \\
&\geq& 2\sum_{i=2}^{n-1}\left( \frac{\left| \lambda_{1}-\frac{\sigma}{n} \right| - \left| \lambda_{i}-\frac{\sigma}{n} \right|+\left| \lambda_{i}-\frac{\sigma}{n} \right| - \left| \lambda_{n}-\frac{\sigma}{n} \right|}{2} \right)^{2}\nonumber    \\
&=&\frac{n-2}{2}\left(\left| \lambda_{1}-\frac{\sigma}{n} \right| - \left| \lambda_{n}-\frac{\sigma}{n} \right|\right)^{2}. 
\end{eqnarray}
Combining (\ref{eq3}), (\ref{eq4}), (\ref{eq5}), the following inequality is presented
\begin{eqnarray}
&&n\left( 2m+\sigma-\frac{\sigma^{2}}{n}\right)-E(G_{S})^{2}  \nonumber  \\
&\geq& \frac{n-2}{2}\left(\left| \lambda_{1}-\frac{\sigma}{n} \right| - \left| \lambda_{n}-\frac{\sigma}{n} \right|\right)^{2}+\left(\left| \lambda_{1}-\frac{\sigma}{n} \right| - \left| \lambda_{n}-\frac{\sigma}{n} \right| \right)^{2}  \nonumber  \\
&=&\frac{n}{2}\left(\left| \lambda_{1}-\frac{\sigma}{n} \right| - \left| \lambda_{n}-\frac{\sigma}{n} \right|\right)^{2}. \nonumber
\end{eqnarray}
After the simplification, we can get bound (\ref{bound2}).
\end{proof}

The upper bound (\ref{bound2}) is better than $E(G_{S})\leq \sqrt{n\left(2m+\sigma-\frac{\sigma^{2}}{n}\right)}$ (\ref{bound0}) when $\left|\lambda_{1}-\frac{\sigma}{n}\right| \neq \left|\lambda_{n}-\frac{\sigma}{n}\right|$. In particular, the upper bound (\ref{bound2}) reduce to Gutmans' bound (\ref{bound0}) when $\left|\lambda_{1}-\frac{\sigma}{n}\right| = \left|\lambda_{n}-\frac{\sigma}{n}\right|$. Besides, the upper bound (\ref{bound2}) is used to characterize all graphs when equality holds in Gutmans' bound (\ref{bound0}).

Let $\rho(A)$ be the $spectral$ $radius$ of a matrix $A$. $\widehat{K_{n}}$ denotes the complete graph $K_{n}$ attaching a self-loop to each of its vertices, $\widetilde{K_{2}}$ denotes $K_{2}$ attaching a self-loop to one of its vertices.

\begin{lemma}  \label{lemma3}
    Let $G_{S}$ be a connected graph with $n \geq 1$ vertices, m edges and $\sigma$ self-loops; $\lambda_{1},\dots,\lambda_{n}$ be the adjacency eigenvalues of $G_{S}$ such that $\left| \lambda_{1}-\frac{\sigma}{n}\right| \geq \dots \geq \left| \lambda_{n}-\frac{\sigma}{n}\right|$. Then $\left|\lambda_{1}-\frac{\sigma}{n}\right| = \left|\lambda_{n}-\frac{\sigma}{n}\right|$ if and only if
\begin{equation*}
  G_{S}=\left\{
  \begin{array}{cc}
   K_{1} \enspace or \enspace K_{2} \enspace & {\rm{if}} \enspace \sigma = 0,\\
   \\
   \widehat{K_{1}} \enspace or \enspace \widetilde{K_{2}} & {\rm{if}} \enspace \sigma = 1,\\
   \\
   \widehat{K_{2}} & {\rm{if}} \enspace \sigma = 2.\\
   \end{array}\right.
\end{equation*}
\end{lemma}

\begin{proof}
The sufficiency of the equality holds is just a straight-forward verification. 

For the necessity part, if $\left|\lambda_{1}-\frac{\sigma}{n}\right| = \left|\lambda_{n}-\frac{\sigma}{n}\right|$, it is easy to see all $\left| \lambda_{i}-\frac{\sigma}{n}\right|$ are equal for $i=1,\dots,n$. Let $\lambda_{1},\dots,\lambda_{n}$ are either $x$ or $\frac{2\sigma}{n}-x$, where $x$ is a constant. Since $G_{S}$ is connected, $A(G_{S})$ is irreducible and non-negative. By the famous Perron-Frobenius Theorem\cite[Theorem 8.4.4.]{Ref14}, the algebraic multiplicity of spectral radius $\rho$ is 1. Assume $x=\rho \geq \frac{2\sigma}{n}-x$, the spectrum of $G_{S}$ is 
      \begin{align*}
        {\rm{Spec}}(G_{S})=\left\{x=\rho,\underbrace{\frac{2\sigma}{n}-x,\dots,\frac{2\sigma}{n}-x}_{n-1}\right\}.
    \end{align*}
By Lemma \ref{lemma1} and direct calculation,
\begin{eqnarray}
\sum_{i=1}^{n}\lambda_{i}(G_{S})=\sigma &\Rightarrow& x+(n-1)\left(\frac{2\sigma}{n}-x\right)=\sigma, \nonumber \\
   &\Rightarrow& (2-n)x=\frac{(2-n)\sigma}{n}. \nonumber
\end{eqnarray}
Now we consider the following 2 cases:

{\bfseries Case 1.} $n=2$. The connected graphs of order $2$ are $K_{2}$, $\widetilde {K_{2}}$ and $\widehat{K_{2}}$, respectively. All of them satisfy the request of the spectrum.

{\bfseries Case 2.} $n \neq 2$. Then $x=\frac{\sigma}{n}$ and all eigenvalues $x$ and $\frac{2\sigma}{n}-x$ are equal to $\frac{\sigma}{n}$.
Again, by the Perron-Frobenius Theorem, the algebraic multiplicity of $\frac{\sigma}{n}$ is 1. Thus $G_{S}=K_{1}$ when $\sigma=0$ or $\widehat{K_{1}}$ when $\sigma=1$.
\end{proof}

By using the Lemma \ref{lemma3}, the main result in this section is presented as follows.

\begin{theorem}
Let $G_{S}$ be a graph with $n\geq1$ vertices, m edges and $\sigma$ self-loops. Then the equality holds in Gutmans' bound (\ref{bound0}) such that
\begin{align} \label{eq6}
E(G_{S}) = \sqrt{n\left(2m+\sigma-\frac{\sigma^{2}}{n}\right)}
\end{align}
if and only if 
\begin{equation*}
G_{S}=\left\{
\begin{array}{cc}
nK_{1} \enspace or \enspace \frac{n}{2}K_{2} & {\rm{if}} \enspace \sigma = 0,\\
\\
\frac{n}{2}K_{1} \cup \frac{n}{2}\widehat{K_{1}} \enspace or \enspace  \frac{n}{2}\widetilde{K_{2}}  & {\rm{if}} \enspace \sigma = \frac{n}{2},\\
\\
n\widehat{K_{1}} \enspace or \enspace \frac{n}{2}\widehat{K_{2}} & {\rm{if}} \enspace \sigma = n.\\
\end{array}\right.
\end{equation*}
\end{theorem}

\begin{proof}
The sufficiency of the case is a simple check. 

For the necessity part, by Theorem \ref{theorem2}, $\left| \lambda_{1}-\frac{\sigma}{n}\right| \geq \dots \geq \left| \lambda_{n}-\frac{\sigma}{n}\right|$ and the upper bound (\ref{bound2}) reduce to Gutmans' bound (\ref{bound0}) when $\left|\lambda_{1}-\frac{\sigma}{n}\right| = \left|\lambda_{n}-\frac{\sigma}{n}\right|$. Thus the equality holds in (\ref{eq6}) only if 
\begin{align*}
    \left| \lambda_{1}-\frac{\sigma}{n}\right|=\dots=\left| \lambda_{n}-\frac{\sigma}{n}\right|.
\end{align*} 
Let $\lambda_{1},\dots,\lambda_{n}$ are either $x$ or $\frac{2\sigma}{n}-x$, where $x$ is a constant. Without loss of generality, assume $x=\rho \geq \frac{2\sigma}{n}-x$ is the spectral radius of $G_{S}$, now we consider the following cases:

{\bfseries Case 1.} $G_{S}$ is connected. By Lemma \ref{lemma3}, $G_{S}$ is either $K_{1}$, $\widehat{K_{1}}$, $K_{2}$, $\widetilde {K_{2}}$ or $\widehat{K_{2}}$ according to the number of self-loops, respectively. All of them satisfy the request of the spectrum.

{\bfseries Case 2.} $G_{S}$ is disconnected. Assume the spectrum of $G_{S}$ is 
\begin{align*}
{\rm{Spec}}(G_{S})=\left\{\underbrace{x,\dots,x}_{p},\underbrace{\frac{2\sigma}{n}-x,\dots,\frac{2\sigma}{n}-x}_{n-p}\right\}.
\end{align*}
By Lemma \ref{lemma1} and direct calculation,
\begin{eqnarray}
\sum_{i=1}^{n}\lambda_{i}(G_{S})=\sigma &\Rightarrow& px+(n-p)\left(\frac{2\sigma}{n}-x\right)=\sigma, \nonumber \\
   &\Rightarrow& (2p-n)x=\frac{(2p-n)\sigma}{n}. \nonumber
\end{eqnarray}

{\bfseries Subcase 2.1.} $n \neq 2p$. Then $x=\frac{\sigma}{n}$ and all eigenvalues $x$ and $\frac{2\sigma}{n}-x$ are equal to $\frac{\sigma}{n}$. By the Perron-Frobenius Theorem, the algebraic multiplicity of all connected components' spectral radius is 1. Therefore $G_{S}$ is comprised of all isolated vertices with $\sigma$ self-loops and $A(G_{S})={\rm{diag}}\{\frac{\sigma}{n},\dots,\frac{\sigma}{n}\}$.
Since the elements of the adjacency matrix $A(G_{S})$ are 0 or 1, $G_{S} = nK_{1}$ when $\sigma=0$; $G_{S} = n\widehat{K_{1}}$ when $\sigma=n$.

{\bfseries Subcase 2.2.} $n = 2p$. Then 
\begin{align*}
{\rm{Spec}}(G_{S})=\left\{\underbrace{x,\dots,x}_{\frac{n}{2}},\underbrace{\frac{2\sigma}{n}-x,\dots,\frac{2\sigma}{n}-x}_{\frac{n}{2}}\right\}.
\end{align*}
Therefore $G_{S}$ contains at least $\frac{n}{2}$ connected components.

{\bfseries Subcase 2.2.1.} $G_{S}$ does not contain $K_{1}$ or $\widehat{K_{1}}$ as a connected component. It follows that the order of all connected components is $2$. The connected component of order 2 are $K_{2}$, $\widetilde{K_{2}}$ or $\widehat{K_{2}}$, respectively. In order to satisfy the spectrum of $G_{S}$, all connected components are identical. Therefore $G_{S}$ is either $\frac{n}{2}K_{2}$, $\frac{n}{2}\widetilde{K_{2}}$ or $\frac{n}{2}\widehat{K_{2}}$ ($n$ is even).

{\bfseries Subcase 2.2.2.} $G_{S}$ contains $K_{1}$ or $\widehat{K_{1}}$ as a connected component.
In this subcase, $x=0$ or $1$. Since $x \geq \frac{2\sigma}{n}-x$, it follows that $G_{S} = nK_{1}$ when $x=0$. When $x=1$, 
\begin{align*}
{\rm{Spec}}(G_{S})=\left\{\underbrace{1,\dots,1}_{\frac{n}{2}},\underbrace{\frac{2\sigma}{n}-1,\dots,\frac{2\sigma}{n}-1}_{\frac{n}{2}}\right\}.
\end{align*}

The characteristic polynomial $\phi(\lambda)$ is a monic polynomial with integer coefficients, so according to the definition of algebraic integer\cite[Section 2.3]{Ref25}, all zeros of $\phi(\lambda)$ are algebraic integer. And it is  known that any rational and algebraic integer number is an integer\cite[Lemma 2.14.]{Ref25}. Thus the rational eigenvalues of $G_{S}$ must be integers. Since $\sigma$ and $n$ are integers, then $\frac{2\sigma}{n}-1$ is rational, so we can conclude that $\sigma=n$, $0$ or $\frac{n}{2}$. 

By using Lemma \ref{lemma1}, $G_{S}=n\widehat{K_{1}}$ when $\sigma=n$; $G_{S}=\frac{n}{2}K_{2}$ when $\sigma=0$, which contradicts the fact that $G_{S}$ contains $K_{1}$ or $\widehat{K_{1}}$ as a connected component;
$G_{S}=\frac{n}{2}K_{1} \cup \frac{n}{2}\widehat{K_{1}}$ when $\sigma=\frac{n}{2}$.
\end{proof}

\section{Some lower bounds for the energy of graphs with self-loops}

In this section, several lower bounds are derived. In addition, the necessary and sufficient conditions are given for the equality holds in some of them.
\begin{lemma} \cite{Ref23} \label{lemma5}
Let $G_{S}$ be a graph obtained from a simple graph $G$ of order $n$ and size $m$ by attaching a self-loop at each vertex of $S \subseteq V(G)$ with $\left| S \right| =\sigma $. If $\left \{\lambda_{i}\right\}_{1\leq i \leq n}$ are the eigenvalues of the adjacency matrix of the graph $G_{S}$, then
\begin{align*}
    \sum_{i<j}\left| \lambda_{i}-\frac{\sigma}{n} \right|\left| \lambda_{j}-\frac{\sigma}{n} \right| \geq m+\frac{\sigma(n-\sigma)}{2n},
\end{align*}
with equality if and only if $G_{S}$ is a totally disconnected graph with $\sigma = 0$ or $\sigma= n$.
\end{lemma}

The following theorem provides lower and upper bounds of $\lambda_{1}(G_{S})$ for any graph with self-loops.

\begin{theorem} \cite{Ref24} 
    Let $G_{S}$ be a graph obtained from a simple graph $G$ of order $n$ and size $m$ by attaching a self-loop at each vertex of $S \subseteq V(G)$ with $\left| S \right| =\sigma $. If $\left \{\lambda_{i}\right\}_{1\leq i \leq n}$ are the eigenvalues of the adjacency matrix of the graph $G_{S}$, then
    \begin{align} \label{bound7}
     \frac{2m+\sigma}{n} \leq  \lambda_{1}(G_{S}) \leq \sqrt{2m+\sigma},
    \end{align}
    and the equality in both inequalities is attained when $G =K_{n}$ and $\sigma=n$.
\end{theorem}

Furthermore, the right-hand side of the inequality (\Ref{bound7}) can be reinforced as a sufficient and necessary condition when equality holds.

\begin{theorem} \label{coro1}
Let $G_{S}$ be a graph with n vertices, m edges and $\sigma$ self-loops; $\lambda_{1} \geq \dots \geq \lambda_{n}$ be the adjacency eigenvalues of $G_{S}$. Then
    \begin{align*}
     \lambda_{1}(G_{S}) \leq \sqrt{2m+\sigma},
    \end{align*}
the equality holds if and only if $G_{S} =\widehat{K_{\sigma}} \cup (n-\sigma)K_{1}$. Moreover, if $G_{S}$ is connected, then $G_{S} =\widehat{K_{n}}$.
\end{theorem}

\begin{proof}
    By Lemma \ref{lemma1}, the following inequality is presented
    \begin{align*}
        \lambda_{1}^{2} \leq \sum_{i=1}^{n}\lambda_{i}^{2}=2m+\sigma,
    \end{align*}
    the equality holds if and only if $\lambda_{1}=\sqrt{2m+\sigma}$ and $\lambda_{i}=0$ for $i=2,\dots,n$. Therefore the spectrum of $G_{S}$ is 
    \begin{align*}
        {\rm{Spec}}(G_{S})=\left\{\sqrt{2m+\sigma},\underbrace{0,\dots,0}_{n-1}\right\}.
    \end{align*}
    
When $G_{S} =\widehat{K_{\sigma}} \cup (n-\sigma)K_{1}$, the sufficiency of the case holds obviously.

For the necessity, since the adjacency matrix $A(G_{S})$ is a real symmetric matrix, its algebraic multiplicity is equal to its geometric multiplicity. Combining above property with the spectrum of $G_{S}$, the rank of the adjacency matrix $A(G_{S})$ is 1. Since $G_{S}$ contains $\sigma$ self-loops and $rank\left(A(G_{S})\right)=1$, the rows(columns) in which the diagonal element is 1 are identical. There is a permutation matrix $P$ such that the adjacency matrix $A(G_{S})$ after permuting has the following form
\begin{align*}
    P^{T}A(G_{S})P=\left[\begin{array}{c|c}
 J_{\sigma \times \sigma } & 0_{\sigma \times (n-\sigma) }\\ \hline 
 0_{(n-\sigma) \times \sigma } & 0_{(n-\sigma) \times (n-\sigma) }
\end{array}\right],
\end{align*}
where $J_{\sigma \times \sigma}$ denote the $\sigma \times \sigma$ matrix whose all entries are 1. 

Therefore, the graph corresponding to the adjacency matrix $A(G_{S})$ is $G_{S} =\widehat{K_{\sigma}} \cup (n-\sigma)K_{1}$. Moreover, if $G_{S}$ is connected, then $G_{S}$ does not contain any isolated vertices. In this case, $\sigma = n$ and $G_{S} =\widehat{K_{n}}$.
\end{proof}

Combining Lemma \ref{lemma5} and Theorem \ref{coro1}, a lower bound is obtained.

\begin{theorem}
    Let $G_{S}$ be a graph with n vertices, m edges and $\sigma$ self-loops; $\lambda_{1} \geq \dots \geq \lambda_{n}$ be the adjacency eigenvalues of $G_{S}$. Then
\begin{align} \label{bound8}
E(G_{S})\geq \sqrt{2\lambda_{1}^{2}-\frac{2\sigma^{2}}{n}},
\end{align}
the equality holds if and only if $G_{S}$ = $nK_{1}$ or $\widehat{K_{1}}$.
\end{theorem}

\begin{proof}[Proof]
By applying Lemma \ref{lemma5} and Corollary \ref{coro1}, the following inequality is given
    \begin{align} 
        E(G_{S})^{2} &= \left(\sum_{i=1}^{n}\left| \lambda_{i}-\frac{\sigma}{n} \right|\right)^{2} = \sum_{i=1}^{n}\left| \lambda_{i}-\frac{\sigma}{n} \right|^{2} + 2\sum_{i<j}\left| \lambda_{i}-\frac{\sigma}{n} \right|\left| \lambda_{j}-\frac{\sigma}{n} \right| \nonumber \\
        &\geq 2m+\sigma-\frac{\sigma^{2}}{n}+2m+\frac{\sigma(n-\sigma)}{n} \nonumber \\
        &= 4m+2\sigma-\frac{2\sigma^{2}}{n} \nonumber \\
        &\geq 2\lambda_{1}^{2}- \frac{2\sigma^{2}}{n}.   \nonumber
  \end{align}
Moreover,  the equality holds in Lemma \ref{lemma5} and Theorem \ref{coro1} if and only if $G_{S}=nK_{1}$ or $n\widehat{K_{1}}$ and $G_{S} =\widehat{K_{\sigma}} \cup (n-\sigma)K_{1}$, respectively. Thus, the equality holds in (\ref{bound8}) if and only if $G_{S}$ = $nK_{1}$ or $\widehat{K_{1}}$. 
\end{proof}

Next, we introduce a special inequality called Ozeki's inequality.

\begin{theorem}\cite[Ozeki's inequality]{Ref15} \label{theorem9}
    Let $a = (a_1, \dots, a_n)$ and $b = (b_1, \dots , b_n)$ be $n\mbox{-}$tuples of real numbers satisfying
    \begin{align*}
        0 \leq m_{1} \leq a_{i} \leq M_{1}  \quad and \quad 0 \leq m_{2} \leq b_{i} \leq M_{2} \quad (i=1,\dots,n),
    \end{align*}
    where $m_{1}=\underset{i=1,\dotsm,n}{\rm{min}} \, a_{i}$, $M_{1}=\underset{i=1,\dotsm,n}{\rm{max}} \, a_{i}$ and $m_{2}=\underset{i=1,\dotsm,n}{\rm{min}} \, b_{i}$, $M_{2}=\underset{i=1,\dotsm,n}{\rm{max}} \, b_{i}$. Then
\begin{align*}
    \left(\sum_{i=1}^{n}a_{i}^{2}\right)\left(\sum_{i=1}^{n}b_{i}^{2}\right)-\left(\sum_{i=1}^{n}a_{i}b_{i}\right)^{2} \leq \frac{n^{2}}{3}\left( M_{1}M_{2}-m_{1}m_{2}\right)^{2}.
\end{align*}
\end{theorem}

By using Ozeki's inequality, a lower bound is presented below.

\begin{theorem} \label{th3}
    Let $G_{S}$ be a graph with n vertices, m edges and $\sigma$ self-loops;  $\lambda_{1},\dots,\lambda_{n}$ be the adjacency eigenvalues of $G_{S}$ such that  
      $  \left| \lambda_{1}-\frac{\sigma}{n}\right| \geq \dots \geq \left| \lambda_{n}-\frac{\sigma}{n}\right|.$
    Then
    \begin{align*}
        E(G_{S})\geq \sqrt{n\left(2m+\sigma-\frac{\sigma^{2}}{n}\right)-\frac{n^{2}}{3}\left(\left| \lambda_{1}-\frac{\sigma}{n} \right| - \left| \lambda_{n}-\frac{\sigma}{n} \right|\right)^{2}}.
    \end{align*}
\end{theorem}

\begin{proof}[Proof]
     By Theorem \ref{theorem9}, setting $a_{i}=\left| \lambda_{i} - \frac{\sigma}{n} \right|$, $b_i=1$, for $i=1,2,\dots,n$. Then $m_{1}=\underset{i=1,\dotsm,n}{\rm{min}} \, a_{i}=\left| \lambda_{n}-\frac{\sigma}{n} \right|$, $M_{1}=\underset{i=1,\dotsm,n}{\rm{max}} \, a_{i}=\left| \lambda_{1}-\frac{\sigma}{n} \right|$, $m_{2}=\underset{i=1,\dotsm,n}{\rm{min}} \, b_{i}=1$, $M_{2}=\underset{i=1,\dotsm,n}{\rm{max}} \, b_{i}=1$ and 
     \begin{align*}
         n\sum_{i=1}^{n} \left| \lambda_{i}-\frac{\sigma}{n} \right|^{2} - \left(\sum_{i=1}^{n}\left| \lambda_{i}-\frac{\sigma}{n}\right| \right)^{2}   \leq  \frac{n^{2}}{3}\left(\left| \lambda_{1}-\frac{\sigma}{n} \right| - \left| \lambda_{n}-\frac{\sigma}{n} \right|\right)^{2}.
    \end{align*} 
    Then by the definition of self-loop graph energy and Lemma \ref{lemma1}, it follows that
      \begin{align*}   
          E(G_{S})^{2} \geq n\left(2m+\sigma-\frac{\sigma^{2}}{n}\right)-\frac{n^{2}}{3}\left(\left| \lambda_{1}-\frac{\sigma}{n} \right| - \left| \lambda_{n}-\frac{\sigma}{n} \right|\right)^{2}.
    \end{align*} 
    That is
          \begin{align*}   
          E(G_{S}) \geq \sqrt{n\left(2m+\sigma-\frac{\sigma^{2}}{n}\right)-\frac{n^{2}}{3}\left(\left| \lambda_{1}-\frac{\sigma}{n} \right| - \left| \lambda_{n}-\frac{\sigma}{n} \right|\right)^{2}}. 
    \end{align*} 
\end{proof}

Let $\mathbf{x} = (x_1,\dots, x_n)$ be an arbitrary $n\mbox{-}$tuple of real numbers and the ultimate energy associated with $\mathbf{x}$ is defined as $UE(\mathbf{x})=\sum_{i=1}^{n}\left|x_{i}-\bar{x}\right|$ by Gutman\cite{Ref8}, where $\bar{x}$ is their arithmetic mean. 

\begin{theorem}\cite{Ref6} \label{lemma11}
    Let $\mathbf{x} = (x_1,\dots, x_n)$ be an arbitrary $n\mbox{-}$tuple of real numbers, where not all $x_1,\dots, x_n$ are equal. Then
\begin{align*}
    UE(\mathbf{x}) \geq \frac{2\sum_{i=1}^{n}(x_{i}-\bar{x})^{2}}{{\rm{max}}_{1\leq i \leq n}x_{i}-{\rm{min}}_{1\leq i \leq n}x_{i}},
\end{align*}
with equality if and only if $x_k = {\rm{max}}_{1\leq i \leq n}x_{i}$ when $x_k > \Bar{x}$ and $x_{k}={\rm{min}}_{1\leq i \leq n}x_{i}$ when $x_k < \bar{x}$.
\end{theorem}

The following Theorem \ref{theorem22} was recently reported in \cite{Ref17} by Liu et al. But it was not mentioned when the equality holds. For completeness, here we provide a different proof by using the notion of ultimate energy and Theorem \ref{lemma11} proposed in \cite{Ref6} by Cai et al.

\begin{theorem} \label{theorem22}
    Let $G_{S}$ be a graph with n vertices, m edges and $\sigma$ self-loops; $\lambda_{1} \geq \dots \geq \lambda_{n}$ be the adjacency eigenvalues of $G_{S}$. Then
\begin{align*}
    E(G_{S}) \geq \frac{4m+2\sigma-\frac{2\sigma^{2}}{n}}{\lambda_{1}-\lambda_{n}},
\end{align*}
the equality holds if and only if $\lambda_k = {\rm{max}}_{1\leq i \leq n}\lambda_{i}$ when $\lambda_k > \frac{\sigma}{n}$ and $\lambda_{k}={\rm{min}}_{1\leq i \leq n}\lambda_{i}$ when $\lambda_k < \frac{\sigma}{n}$.
\end{theorem}

\begin{proof}[Proof]
      By Theorem \ref{lemma11}, let $\mathbf{x} = (\lambda_1,\dots, \lambda_n)$, then $UE(\mathbf{x})=E(G_{S})=\sum_{i=1}^{n}\left|\lambda_{i}-\frac{\sigma}{n}\right|$. Let $\lambda_{1} = {\rm{max}}_{1\leq i \leq n}\lambda_{i}$, $\lambda_{n}={\rm{min}}_{1\leq i \leq n}\lambda_{i}$, then by Lemma \ref{lemma1}, we have the following inequality
      \begin{align*}
          E(G_{S}) \geq \frac{2\sum_{i=1}^{n}\left|\lambda_{i}-\frac{\sigma}{n}\right|^{2}}{\lambda_{1}-\lambda_{n}} = \frac{4m+2\sigma-\frac{2\sigma^{2}}{n}}{\lambda_{1}-\lambda_{n}},
      \end{align*}
the equality holds if and only if $\lambda_{k}=\lambda_{1}$ if $\lambda_{k} > \frac{\sigma}{n}$ and $\lambda_{k}=\lambda_{n}$ if $\lambda_{k} < \frac{\sigma}{n}$.
\end{proof}

\section*{Acknowledgments}
This work was partially supported by the National Natural Science Foundation of China (No.12171088).

\end{document}